\documentclass[11pt]{article}

\usepackage{amsmath,amssymb,amsfonts,mathtools,amsthm}
\usepackage{geometry}
\usepackage{float}
\usepackage[hidelinks]{hyperref}
\usepackage{enumitem} 
\usepackage{graphicx}
\usepackage{graphicx}
\newcommand{\bu}{\mathbf{u}} 
\newcommand{\ep}{\varepsilon} 
\theoremstyle{plain}
\newtheorem{theorem}{Theorem}
\newtheorem{lemma}{Lemma}

\newtheorem{corollary}{Corollary}

\theoremstyle{remark}
\newtheorem{remark}{Remark}

\newcommand{\R}{\mathbb{R}}
\newcommand{\bF}{\overline{F}}
\newcommand{\dd}{\,\mathrm{d}}

\title{Right-tail asymptotics for products of independent normal random variables}
\author{D\v{z}iugas Chvoinikov and Jonas \v{S}iaulys}
\date{}

\begin{document}
\maketitle

\begin{abstract}
Let $X_1,\dots,X_n$ be independent normal random variables with $X_i\sim N(\mu_i,\sigma_i^2)$, and set $Z=\prod_{i=1}^n X_i$.
We derive asymptotic approximations for the right tail probability $\mathbb{P}(Z>x)$ as $x\to\infty$.
When at least one mean is nonzero, the asymptotic formula remains explicit and involves a finite multiplicative factor arising from admissible sign patterns (reflecting the different ways the product can be positive); it includes an explicit first relative correction term of order $x^{-1/n}$, with remaining relative error $O(x^{-2/n})$.
The proof uses a boundary saddle-point/Laplace method: first a multidimensional Laplace approximation near the boundary saddle, then a one-dimensional endpoint Laplace approximation.
\end{abstract}


\section{Introduction}

Products of random variables appear naturally in many settings. A simple example is compound growth:
if successive multiplicative factors are random (e.g.\ one-period returns), then the total factor is a product.
Products also arise in physics and engineering models where measured quantities are formed by multiplying
noisy components. Because of this, it is useful to understand the distribution and tail behaviour of products.

Even for normal random variables, the exact distribution of a product can be complicated.
For two independent standard normal variables $\xi$ and $\eta$, the density of $\xi\eta$ can be written using a modified Bessel function; see~\cite{Craig1936,WishartBartlett1932,Yuan1933}.
There are also exact formulas for products of correlated normal variables; see~\cite{Craig1936,CuiYuIommelliKong2016,Gaunt2019,NadarajahPogany2016}.
More recently, exact density formulas were also obtained in~\cite{Gaunt2021}.
For products of more than two independent normal variables, exact densities can be expressed in terms of
special functions (e.g.\ Meijer $G$-functions), see Springer and Thompson~\cite{SpringerThompson1966,SpringerThompson1970}.

Alongside exact distributional formulas, asymptotic approximations for tail probabilities are also of interest.
In the zero-mean case, Leipus, \v{S}iaulys, Dirma and Zov\'e~\cite{LeipusSiaulysDirmaZove2023}
derived a precise right-tail asymptotic for the product of $n$ independent zero-mean normal variables;
see also Arendarczyk and Debicki~\cite{ArendarczykDebicki2011} for a related Weibull-tail asymptotic result.

The present paper focuses on the right tail of the product
\[
Z=\prod_{i=1}^n X_i,\qquad X_i\sim N(\mu_i,\sigma_i^2)\ \text{independent},
\]
in the case where at least one mean $\mu_i$ is nonzero.
Our main result (Theorem~\ref{thm:munonzero}) gives an explicit asymptotic approximation for
$\mathbb{P}(Z>x)$ as $x\to\infty$, including the explicit first relative correction term of order $x^{-1/n}$, with remaining relative error $O(x^{-2/n})$.
The formula remains simple to evaluate: it reduces to a finite computation over admissible sign patterns,
summarized by the quantities $L_*$ and $m_*$, and Remark~\ref{rem:Lstar} provides an $O(n)$ procedure
to compute them.

The proof uses standard Laplace and saddle-point arguments. We first show that the main contribution comes from the balanced region, then perform a multidimensional Laplace approximation near a boundary saddle, and finally apply a one-dimensional endpoint Laplace approximation to the remaining integral. Our technical tools are taken from Wong's treatment of Laplace-type
integrals~\cite{Wong2001}.

\paragraph{Structure of the paper.}
Section~\ref{sec:tools_lemmas} states the Laplace-method tools used in the proof.
Section~\ref{sec:saddle} derives and analyzes the boundary saddle system.
The remaining sections compute the saddle expansion and prefactor and complete the proof of
Theorem~\ref{thm:munonzero}.

\section{Main results}

Let $X_i\sim N(\mu_i,\sigma_i^2)$ be independent, with $\sigma_i>0$, and define
\[
Z:=\prod_{i=1}^n X_i,
\qquad
\bF_n(x):=\mathbb{P}(Z> x),
\qquad x\to+\infty.
\]

\begin{theorem}\label{thm:munonzero}
Assume that at least one $\mu_i$ is nonzero. Define
\[
\mathcal S:=\Big\{s=(s_1,\dots,s_n)\in\{\pm1\}^n:\ \prod_{i=1}^n s_i=+1\Big\},
\qquad
L_s:=\sum_{i=1}^n s_i\frac{\mu_i}{\sigma_i},
\]
\[
L_*:=\max_{s\in\mathcal S}L_s,
\qquad
\mathcal S_*:=\{s\in\mathcal S:\ L_s=L_*\},
\qquad
m_*:=|\mathcal S_*|,
\qquad
C:=\exp\!\left(-\sum_{i=1}^n\frac{\mu_i^2}{2\sigma_i^2}\right).
\]
Then, for any $s\in\mathcal S_*$, as $x\to\infty$,
\begin{align*}
\bF_n(x)
=
\frac{C}{2^{n/2}\sqrt{\pi n}}
\left(\frac{\prod_{j=1}^n \sigma_j}{x}\right)^{1/n}
m_*\,
\exp\!\Bigg\{
-\frac{n}{2}\,\left(\frac{x}{\prod_{j=1}^n \sigma_j}\right)^{2/n}+L_*\left(\frac{x}{\prod_{j=1}^n \sigma_j}\right)^{1/n}
+\frac{1}{4}\left(
\sum_{i=1}^n\left(\frac{\mu_i}{\sigma_i}\right)^2
-\frac{1}{n}L_*^2
\right) \\
+\frac{1}{16}\left(\frac{\prod_{j=1}^n \sigma_j}{x}\right)^{1/n}
\left(
\sum_{i=1}^n s_i\left(\frac{\mu_i}{\sigma_i}\right)^3
+\frac{L_*}{n}\sum_{i=1}^n\left(\frac{\mu_i}{\sigma_i}\right)^2
-\frac{2}{n^2}L_*^3
\right)
\Bigg\}
\left(
1+\frac{n+3}{4n}\,L_*
\left(\frac{\prod_{j=1}^n \sigma_j}{x}\right)^{1/n}
+O(x^{-2/n})
\right).
\end{align*}
\end{theorem}

\begin{remark}[Computing $L_*$ and $m_*$]\label{rem:Lstar}
Let $a_i:=\mu_i/\sigma_i$ and recall $L_s=\sum_{i=1}^n s_i a_i$ with $s_i\in\{\pm1\}$ and $\prod_{i=1}^n s_i=+1$.
Let
\[
I_0:=\{i:\ a_i=0\},\qquad k:=|I_0|.
\]
For $i\notin I_0$ set the sign
\[
s_i^{(0)}:=\begin{cases}
+1,& a_i>0,\\
-1,& a_i<0,
\end{cases}
\]
so that $s_i^{(0)}a_i=|a_i|$ for all $i\notin I_0$. Let
\[
p_0:=\prod_{i\notin I_0} s_i^{(0)}\in\{\pm1\}.
\]

\begin{itemize}
\item \textbf{If $k\ge 1$:} fix the signs on the nonzero coordinates by setting
$s_i:=s_i^{(0)}$ for all $i\notin I_0$, so that $\sum_{i\notin I_0}s_i a_i=\sum_{i\notin I_0}|a_i|$.
Now choose the signs $\{s_i\}_{i\in I_0}$ so that
\[
\prod_{i\in I_0} s_i = p_0,
\]
which ensures the overall constraint
\[
\prod_{i=1}^n s_i=\Big(\prod_{i\notin I_0} s_i^{(0)}\Big)\Big(\prod_{i\in I_0} s_i\Big)=p_0\cdot p_0=+1.
\]
This does not change the objective since $a_i=0$ for $i\in I_0$. Hence
\[
L_*=\sum_{i=1}^n |a_i|,
\qquad
m_*=2^{k-1}.
\]
Choose $k-1$ signs in $I_0$ arbitrarily; the last one is then uniquely determined by $\prod_{i\in I_0}s_i=p_0$, so $m_*=2^{k-1}$.

\item \textbf{If $k=0$ and $p_0=+1$:} the pattern $s=s^{(0)}$ is feasible and yields
\[
L_*=\sum_{i=1}^n |a_i|,
\qquad
m_*=1.
\]

\item \textbf{If $k=0$ and $p_0=-1$:} feasibility forces flipping an odd number of signs; flipping more than one strictly decreases $L_s$ further, so the maximizers flip exactly one coordinate.
Flipping index $j$ reduces $L_s$ by $2|a_j|$, so one should flip an index attaining $\min_{1\le i\le n}|a_i|$. Thus
\[
L_*=\sum_{i=1}^n |a_i|-2\min_{1\le i\le n}|a_i|,
\qquad
m_*=\bigl|\{j:\ |a_j|=\min_i|a_i|\}\bigr|.
\]
\end{itemize}

The above computation runs in linear time in $n$.
\end{remark}

\section{Tools and lemmas}
\label{sec:tools_lemmas}

We will use two standard Laplace-method facts: a multidimensional Laplace approximation
and an endpoint (maximum at the boundary) Laplace approximation; see Wong~\cite{Wong2001}.

\begin{lemma}[Laplace approximation in $\R^n$]\label{lem:wong_multi}
Let $D\subset\R^n$ be a (possibly unbounded) domain and let
\[
J(\lambda):=\int_D g(x)\,e^{-\lambda f(x)}\,dx,
\]
where $\lambda$ is a large positive parameter. Assume $f$ and $g$ are smooth on $D$ and that:

\begin{enumerate}[label=(\roman*)]
\item $J(\lambda)$ converges absolutely for all $\lambda\ge \lambda_0$.
\item For every $\varepsilon>0$, $\rho(\varepsilon)>0$, where
\[
\rho(\varepsilon):=\inf\{\,f(x)-f(x_0): x\in D \text{ and } \|x-x_0\|\ge \varepsilon\,\}.
\]
Condition~(ii) implies that $f$ attains its minimum at, and only at, the point $x_0$.
If $x_0$ is an interior point of $D$, then $x_0$ is a critical point of $f$, i.e.\ $\nabla f(x_0)=0$.
\item The Hessian matrix at $x_0$,
\[
A=\Big(\frac{\partial^2 f}{\partial x_i\partial x_j}\Big)\Big|_{x=x_0},
\]
is positive definite.
\end{enumerate}
If $x_0$ is an interior point of $D$, then as $\lambda\to\infty$,
\[
J(\lambda)\sim \left(\frac{2\pi}{\lambda}\right)^{n/2} g(x_0)\,(\det A)^{-1/2}\,e^{-\lambda f(x_0)}.
\]
\end{lemma}

\begin{corollary}\label{cor:wong_multi_rate}
Under the assumptions of Lemma~\ref{lem:wong_multi}, assume in addition that $f$ and $g$ are infinitely differentiable. If $x_0$ is an interior point of $D$, then
\[
J(\lambda)=\left(\frac{2\pi}{\lambda}\right)^{n/2} g(x_0)\,(\det A)^{-1/2}\,e^{-\lambda f(x_0)}
\Big(1+\frac{\kappa}{\lambda}+O(\lambda^{-2})\Big),
\qquad \lambda\to\infty,
\]
where
\[
\kappa:=\frac{c_1}{c_0},
\]
and $c_0,c_1$ are the first two coefficients in Wong's asymptotic expansion
\[
J(\lambda)\sim e^{-\lambda f(x_0)}\sum_{k=0}^\infty c_k\,\lambda^{-n/2-k}.
\]
\end{corollary}

\begin{proof}
Wong's Theorem~3 gives an asymptotic expansion
\[
J(\lambda)\sim e^{-\lambda f(x_0)}\sum_{k=0}^\infty c_k\,\lambda^{-n/2-k}.
\]
Truncating after $k=1$ yields
\[
J(\lambda)
=
e^{-\lambda f(x_0)}
\left(c_0\lambda^{-n/2}+c_1\lambda^{-n/2-1}+O(\lambda^{-n/2-2})\right).
\]
Factoring out the leading term,
\[
J(\lambda)
=
c_0\lambda^{-n/2}e^{-\lambda f(x_0)}
\left(1+\frac{c_1}{c_0}\lambda^{-1}+O(\lambda^{-2})\right).
\]
Moreover,
\[
c_0=(2\pi)^{n/2}g(x_0)(\det A)^{-1/2},
\]
so the stated formula follows with $\kappa=c_1/c_0$.
\end{proof}
The corresponding quantity $\kappa$ will be derived later.

\noindent\textbf{Application to our inner integral.}
For each sign region $s$, we apply Corollary~\ref{cor:wong_multi_rate} with $n$ replaced by $n-1$ and
\[
D=D_s,\qquad x=\tilde{\bu},\qquad \lambda=r(w),\qquad
f(\tilde{\bu})=\frac{1}{r(w)}\Phi_w(\tilde{\bu}),\qquad
g(\tilde{\bu})=\frac{1}{|u_1\cdots u_{n-1}|}.
\]
Then $x_0=\tilde{\bu}_s(w)$, $f(x_0)=S_s(w)/r(w)$, and
\[
\det A = r(w)^{-(n-1)}\det H_s(w),
\]
so Corollary~\ref{cor:wong_multi_rate} yields, as $w\to\infty$,
\[
\int_{D_s}\exp\!\big(-\Phi_w(\tilde{\bu})\big)\,\frac{d\tilde{\bu}}{|u_1\cdots u_{n-1}|}
=
(2\pi)^{\frac{n-1}{2}}\,
\frac{e^{-S_s(w)}}{|u_{1,s}(w)\cdots u_{n-1,s}(w)|\,\sqrt{\det H_s(w)}}
\Big(1+\frac{\kappa_s(w)}{r(w)}+O(r(w)^{-2})\Big),
\]
where $\kappa_s(w)$ denotes the first correction coefficient coming from the multidimensional Laplace expansion in this application.

\noindent\textbf{Verification of the assumptions of Lemma~\ref{lem:wong_multi}.}
From step 1 and 2 in prefactor section. Fix an admissible sign region $s$ and $w$ large. Recall that $\Phi_w$ is smooth on each sign region $D_s$
and that $f(\tilde{\bu})=\Phi_w(\tilde{\bu})/r(w)$ with $r(w)>0$.

\emph{(i) assumption.}
Write $u_n=w/(u_1\cdots u_{n-1})$ and recall
\[
\Phi_w(u_1,\dots,u_{n-1})
=
\sum_{k=1}^{n-1}\left(\frac{u_k^2}{2\sigma_k^2}-\frac{\mu_k}{\sigma_k^2}u_k\right)
+\left(\frac{u_n^2}{2\sigma_n^2}-\frac{\mu_n}{\sigma_n^2}u_n\right).
\]
Using $a^2/2-ab\ge a^2/4-b^2$ (with $a=u_k/\sigma_k$, $b=\mu_k/\sigma_k$) gives
\[
\frac{u_k^2}{2\sigma_k^2}-\frac{\mu_k}{\sigma_k^2}u_k \;\ge\;
\frac{u_k^2}{4\sigma_k^2}-\frac{\mu_k^2}{\sigma_k^2},
\qquad
\frac{u_n^2}{2\sigma_n^2}-\frac{\mu_n}{\sigma_n^2}u_n \;\ge\;
\frac{u_n^2}{4\sigma_n^2}-\frac{\mu_n^2}{\sigma_n^2}.
\]
Hence there exist constants $c_0,C_0>0$ (independent of $\tilde{\bu}$) such that
\[
\Phi_w(\tilde{\bu})\ge c_0\Big(\sum_{k=1}^{n-1}u_k^2 + u_n^2\Big)-C_0
\quad\Longrightarrow\quad
e^{-\Phi_w(\tilde{\bu})}\le C\,\exp\!\Big(-c_0\sum_{k=1}^{n-1}u_k^2\Big)\,
\exp\!\Big(-c_0\,\frac{w^2}{(u_1\cdots u_{n-1})^2}\Big),
\]
for some $C>0$.

Split $D_s=E_1\cup E_2$, where
$E_1:=\{\tilde{\bu}\in D_s:\ |u_k|\ge 1\ \forall k\}$ and $E_2:=D_s\setminus E_1$.
On $E_1$ we have $|u_1\cdots u_{n-1}|\ge 1$, hence
\[
\frac{e^{-\Phi_w(\tilde{\bu})}}{|u_1\cdots u_{n-1}|}
\le C\,\exp\!\Big(-c_0\sum_{k=1}^{n-1}u_k^2\Big),
\]
which is integrable on $\R^{n-1}$.

On $E_2$ we have $z:=|u_1\cdots u_{n-1}|\le 1$. Recall that from the lower bound on $\Phi_w$ we obtained
\[
e^{-\Phi_w(\tilde{\bu})}
\le
C\,\exp\!\Big(-c_0\sum_{k=1}^{n-1}u_k^2\Big)\,
\exp\!\Big(-c_0\,\frac{w^2}{(u_1\cdots u_{n-1})^2}\Big),
\]
hence
\[
\frac{e^{-\Phi_w(\tilde{\bu})}}{|u_1\cdots u_{n-1}|}
\le
C\,\exp\!\Big(-c_0\sum_{k=1}^{n-1}u_k^2\Big)\,
\frac{1}{z}\exp\!\Big(-c_0\,\frac{w^2}{z^2}\Big).
\]
Let $A:=c_0w^2>0$. Using $e^{x}\ge x$ for $x>0$, we have $e^{-x}\le 1/x$, and therefore
\[
\exp\!\Big(-\frac{A}{z^2}\Big)\le \frac{z^2}{A},\qquad 0<z\le 1.
\]
Consequently,
\[
\frac{1}{z}\exp\!\Big(-\frac{A}{z^2}\Big)
\le
\frac{1}{z}\cdot \frac{z^2}{A}
=
\frac{z}{A}
\le
\frac{1}{A}
\qquad(0<z\le 1).
\]
Thus on $E_2$,
\[
\frac{e^{-\Phi_w(\tilde{\bu})}}{|u_1\cdots u_{n-1}|}
\le
\frac{C}{A}\,\exp\!\Big(-c_0\sum_{k=1}^{n-1}u_k^2\Big).
\]
The right-hand side is integrable on $\R^{n-1}$, hence in particular integrable on $E_2$.
Together with the estimate on $E_1$, this yields
\[
\int_{D_s}\exp\!\big(-\Phi_w(\tilde{\bu})\big)\,\frac{d\tilde{\bu}}{|u_1\cdots u_{n-1}|}<\infty.
\]

\emph{(ii) assumption.} From (i) (dividing by $r(w)$) we have
\[
f(\tilde{\bu})\ge \frac{c_0}{r(w)}\Big(\sum_{k=1}^{n-1}u_k^2+\frac{w^2}{(u_1\cdots u_{n-1})^2}\Big)-\frac{C_0}{r(w)}.
\]
In particular, $f(\tilde u)\to\infty$ if either $\max_k|u_k|\to\infty$ or $|u_1\cdots u_{n-1}|\to0$, so for every $M>0$
the set $\{\tilde u\in D_s:\ f(\tilde u)\le M\}$ is compact. Since $f$ is continuous on $D_s$, it attains its minimum
on $D_s$ at some point $x_0\in D_s$.
Since $x_0$ is an interior minimizer and $f$ is $C^1$ on $D_s$, we have $\nabla f(x_0)=0$.
By the saddle-point analysis (Section~\ref{sec:saddle}), $f(\tilde{\bu})$ has a unique minimizer in $D_s$,
namely $\tilde u_s(w)$; therefore $x_0=\tilde u_s(w)$ and the minimizer is unique.

Fix $\ep>0$ and set $A_\ep:=\{\tilde{\bu}\in D_s:\ \|\tilde{\bu}-x_0\|\ge\ep\}$.
By the bound above, there exist $R,\delta>0$ such that
\[
\inf_{A_\ep} f \;=\; \inf_{A_\ep\cap K_{R,\delta}} f,
\qquad
K_{R,\delta}:=\Big\{\tilde{\bu}\in D_s:\ \max_k|u_k|\le R,\ |u_1\cdots u_{n-1}|\ge\delta\Big\}.
\]
Since $A_\ep\cap K_{R,\delta}$ is closed and bounded, and $f$ is continuous, $f$ attains its minimum on
$A_\ep\cap K_{R,\delta}$; denote this minimum by $m_\ep$.
Uniqueness of the minimizer implies $m_\ep>f(x_0)$, and therefore
\[
\rho(\ep):=\inf\{\,f(\tilde{\bu})-f(x_0):\tilde{\bu}\in D_s,\ \|\tilde{\bu}-x_0\|\ge\ep\,\}
=m_\ep-f(x_0)>0.
\]

\emph{(iii) assumption.}
Let $H_s(w)=\nabla^2\Phi_w(\tilde{\bu}_s(w))$ in the variables $(u_1,\dots,u_{n-1})$. From Prefactor Step~2(b),
\[
(H_s(w))_{ii}=\frac{4}{\sigma_i^2}\Big(1+O(r(w)^{-1})\Big),\qquad
(H_s(w))_{ij}=\frac{2s_is_j}{\sigma_i\sigma_j}\Big(1+O(r(w)^{-1})\Big)\quad(i\neq j).
\]
First ignore the $O(r(w)^{-1})$ terms and denote by $H_0$ the leading matrix with
$(H_0)_{ii}=4/\sigma_i^2$ and $(H_0)_{ij}=2s_is_j/(\sigma_i\sigma_j)$.
For any $x=(x_1,\dots,x_{n-1})\neq 0$, set
\[
y_i:=\frac{s_i x_i}{\sigma_i}\qquad (1\le i\le n-1),
\]
so $y\neq 0$. Expanding $x^\top H_0 x$ gives,
\[
x^\top H_0 x=\sum_{i=1}^{n-1}\frac{4}{\sigma_i^2}x_i^2
+2\sum_{i<j}\frac{2s_is_j}{\sigma_i\sigma_j}x_ix_j
=4\sum_{i=1}^{n-1}y_i^2+4\sum_{i<j}y_i y_j.
\]
By the identity $\big(\sum_i y_i\big)^2=\sum_i y_i^2+2\sum_{i<j}y_i y_j$, we get
\[
4\sum_i y_i^2+4\sum_{i<j}y_i y_j
=2\sum_i y_i^2+2\Big(\sum_i y_i\Big)^2>0,
\]
so $x^\top H_0 x>0$ for all $x\neq 0$, i.e.\ $H_0$ is positive definite. Since
$H_s(w)=H_0+O(r(w)^{-1})$ entrywise, the error is small for $w$ large and the inequality
$x^\top H_s(w)x>0$ remains true; hence $H_s(w)$ is positive definite for $w$ large.
Finally, since $f=\Phi_w/r(w)$ with $r(w)>0$,
\[
\nabla^2 f(x_0)=\frac{1}{r(w)}\,H_s(w)
\]
is positive definite as well.

\begin{lemma}[Endpoint Laplace expansion at a boundary minimum]\label{lem:wong_thm1}
Let
\[
I(\lambda)=\int_a^b \varphi(x)\,e^{-\lambda h(x)}\,dx,
\qquad \lambda\to\infty,
\]
where $a<b\le\infty$ and $\lambda>0$. Assume:
\begin{enumerate}[label=(\roman*)]
\item $h(x)>h(a)$ for all $x\in(a,b)$ and, for every $\delta>0$,
\[
\inf_{x\in[a+\delta,b)}\bigl(h(x)-h(a)\bigr)>0.
\]
\item $h'(x)$ and $\varphi(x)$ are continuous in a neighborhood of $x=a$ (allowing a possible exception at $x=a$).
\item As $x\to a^+$, there exist constants $\mu>0$ and $\alpha$ with $\Re(\alpha)>0$, and coefficients
$\{a_s\}_{s\ge 0}$, $\{b_s\}_{s\ge 0}$ with $a_0\neq 0$ and $b_0\neq 0$, such that
\[
h(x)\sim h(a)+\sum_{s=0}^\infty a_s (x-a)^{s+\mu},
\qquad
\varphi(x)\sim \sum_{s=0}^\infty b_s (x-a)^{s+\alpha-1},
\]
and the expansion for $h(x)$ may be differentiated termwise, giving
\[
h'(x)\sim \sum_{s=0}^\infty a_s(s+\mu)\,(x-a)^{s+\mu-1}.
\]
\item $I(\lambda)$ converges absolutely for all sufficiently large $\lambda$.
\end{enumerate}
Then there exist coefficients $\{c_s\}_{s\ge 0}$ (expressible in terms of $\{a_s\}$ and $\{b_s\}$) such that
\[
I(\lambda)\sim e^{-\lambda h(a)}\sum_{s=0}^\infty
\Gamma\!\Big(\frac{s+\alpha}{\mu}\Big)\,\frac{c_s}{\lambda^{(s+\alpha)/\mu}},
\qquad \lambda\to\infty,
\]
and in particular the first two coefficients are
\[
c_0=\frac{b_0}{\mu\,a_0^{\alpha/\mu}},
\qquad
c_1=
\left(
\frac{b_1}{\mu}
-\frac{(\alpha+1)a_1b_0}{\mu^2 a_0}
\right)
\frac{1}{a_0^{(\alpha+1)/\mu}}.
\]
\end{lemma}

\begin{proof}
This is a restatement of Wong's Theorem~1 with the same hypotheses and notation.
\end{proof}

\begin{corollary}[Endpoint rule for $\mu=\alpha=1$]
\label{cor:wong_mu1_alpha1_rate}
Under the assumptions of Lemma~\ref{lem:wong_thm1}, if $\mu=\alpha=1$, then
\[
I(\lambda)= \frac{\varphi(a)}{\lambda h'(a)}e^{-\lambda h(a)}
\Bigl(1+\frac{\eta}{\lambda}+O(\lambda^{-2})\Bigr),
\qquad \lambda\to\infty,
\]
where
\[
\eta:=\frac{c_1}{c_0},
\]
and $c_0,c_1$ are the first two coefficients in the asymptotic expansion from
Lemma~\ref{lem:wong_thm1}.
\end{corollary}

\begin{proof}
With $\mu=\alpha=1$, Lemma~\ref{lem:wong_thm1} gives
\[
I(\lambda)\sim e^{-\lambda h(a)}\sum_{s=0}^\infty \Gamma(s+1)c_s\,\lambda^{-(s+1)}.
\]
Taking the first three terms,
\[
I(\lambda)
= e^{-\lambda h(a)}
\left(
\frac{c_0}{\lambda}
+\frac{c_1}{\lambda^2}
+O(\lambda^{-3})
\right).
\]
Factoring out the leading term gives
\[
I(\lambda)
=
\frac{c_0}{\lambda}e^{-\lambda h(a)}
\Bigl(1+\frac{c_1}{c_0}\lambda^{-1}+O(\lambda^{-2})\Bigr).
\]

Moreover, for $\mu=\alpha=1$, the expansions in Lemma~\ref{lem:wong_thm1} become
\[
h(x)\sim h(a)+a_0(x-a)+a_1(x-a)^2+\cdots,
\qquad
\varphi(x)\sim b_0+b_1(x-a)+\cdots.
\]
Thus
\[
a_0=h'(a),\qquad a_1=\frac{h''(a)}{2},
\qquad
b_0=\varphi(a),\qquad b_1=\varphi'(a).
\]
Hence
\[
c_0=\frac{b_0}{a_0}=\frac{\varphi(a)}{h'(a)}.
\]
Also, Wong's formula for $c_1$ yields
\[
c_1=\frac{b_1}{a_0^2}-\frac{2a_1b_0}{a_0^3}
=\frac{\varphi'(a)}{h'(a)^2}-\frac{\varphi(a)h''(a)}{h'(a)^3}.
\]
Therefore
\[
I(\lambda)= \frac{\varphi(a)}{\lambda h'(a)}e^{-\lambda h(a)}
\Bigl(1+\frac{\eta}{\lambda}+O(\lambda^{-2})\Bigr),
\]
where
\[
\eta=\frac{c_1}{c_0}
=\frac{\varphi'(a)}{\varphi(a)h'(a)}-\frac{h''(a)}{h'(a)^2}.
\]
\end{proof}

\noindent\textbf{Application.}
Fix $s$ and write $w=xt$. Then
\[
\int_x^\infty A_s(w)e^{-S_s(w)}\,dw
= x\int_1^\infty A_s(xt)\,e^{-S_s(xt)}\,dt.
\]
Set $\lambda:=r(x)^2$ and define
\[
h_x(t):=\frac{S_s(xt)}{\lambda},\qquad \varphi_x(t):=xA_s(xt).
\]
Then $\lambda h_x(1)=S_s(x)$ and $\lambda h_x'(1)=xS_s'(x)$. Assuming the hypotheses of
Lemma~\ref{lem:wong_thm1} hold for $h_x,\varphi_x$ at $t=1$ with $\mu=\alpha=1$, Corollary~\ref{cor:wong_mu1_alpha1_rate}
yields
\[
\int_x^\infty A_s(w)e^{-S_s(w)}\,dw
=
\frac{A_s(x)}{S_s'(x)}\,e^{-S_s(x)}
\Bigl(1+\frac{\eta_{s}(x)}{\lambda}+O(\lambda^{-2})\Bigr),
\qquad \lambda=r(x)^2,
\]
where
\[
\eta_s(x)
=
\frac{\varphi_x'(1)}{\varphi_x(1)\,h_x'(1)}
-
\frac{h_x''(1)}{h_x'(1)^2}.
\]

\noindent\textbf{Verification of Lemma 2 assumptions (case $\mu=\alpha=1$).}
From Step~3 in the Prefactor section, fix an admissible sign pattern $s$ and set $\lambda:=r(x)^2$,
$h_x(t):=S_s(xt)/\lambda$, and $\varphi_x(t):=xA_s(xt)$ on $t\in[1,\infty)$.

\emph{(i) assumption.}We have, as $w\to\infty$,
\[
S_s'(w)=\frac{1}{\big(\prod_{j=1}^n\sigma_j\big)\,r(w)^{\,n-2}}\Big(1+O(r(w)^{-1})\Big)>0,
\]
hence $S_s'(w)>0$ for all $w$ large enough. Therefore, for $x$ large and all $t\ge 1$,
\[
h_x'(t)=\frac{1}{\lambda}\frac{d}{dt}S_s(xt)=\frac{x}{\lambda}S_s'(xt)>0,
\qquad \lambda=r(x)^2,
\]
so $h_x$ is strictly increasing on $[1,\infty)$ and attains its unique minimum at $t=1$.

Fix $\delta>0$. By monotonicity,
\[
\inf_{t\ge 1+\delta}\bigl(h_x(t)-h_x(1)\bigr)=h_x(1+\delta)-h_x(1),
\]
and the saddle expansion gives
\[
h_x(1+\delta)-h_x(1)
=\frac{n}{2}\big((1+\delta)^{2/n}-1\big)+O(r(x)^{-1})>0
\]
for all $x$ large, which verifies hypothesis \textup{(i)}.

\emph{(ii) assumption.}
Since $S_s$ and $A_s$ are obtained by evaluating smooth functions at the minimizer, they are continuous.
Therefore $h_x'$ and $\varphi_x$ are continuous near $t=1$.

\emph{(iii) assumption.}
By differentiability of $S_s$ and continuity of $A_s$,
\[
h_x(t)=h_x(1)+h_x'(1)(t-1)+O((t-1)^2),
\qquad
\varphi_x(t)=\varphi_x(1)+O(t-1),
\qquad (t\downarrow 1).
\]
 \quad\text{where}\quad
\[
h_x'(1)=\frac{x}{\lambda}S_s'(x)
=\frac{x}{r(x)^2}\cdot \frac{1}{\big(\prod_{j=1}^n\sigma_j\big)\,r(x)^{n-2}}\Big(1+O(r(x)^{-1})\Big)
=1+O\!\big(r(x)^{-1}\big)>0,
\]
for $x$ large. Thus we are in the case $\mu=\alpha=1$.

\emph{(iv) assumption.}
Fix $s$ and recall $r(w):=\big(w/\prod_{j=1}^n\sigma_j\big)^{1/n}$.
From \eqref{eq:tangent_laplace_n} we have, for $w$ large,
\[
A_s(w)=\frac{\pi^{(n-1)/2}}{\sqrt n\,r(w)^{\,n-1}}\Big(1+O(r(w)^{-1})\Big),
\]
hence there exist constants $C_A>0$ and $W_1$ such that
\[
0\le A_s(w)\le \frac{C_A}{r(w)^{\,n-1}},\qquad w\ge W_1.
\]
Moreover, the saddle expansion gives
\[
S_s(w)=\frac n2\,r(w)^2 - L_s r(w) + O(1),
\]
so there exist $c_S>0$ and $W_2$ such that
\[
S_s(w)\ge c_S\,r(w)^2,\qquad w\ge W_2.
\]
Let $W:=\max\{W_1,W_2\}$. Then for $x\ge W$,
\[
0\le \int_x^\infty A_s(w)e^{-S_s(w)}\,dw
\le C_A\int_x^\infty r(w)^{-(n-1)}e^{-c_S r(w)^2}\,dw
\le C_A\int_x^\infty e^{-c_S r(w)^2}\,dw.
\]
Since $r(w)^2=\big(w/\prod\sigma_j\big)^{2/n}$, the last integral is bounded:
\[
\int_x^\infty e^{-c_S r(w)^2}\,dw
=
\int_x^\infty \exp\!\left(-c_S\left(\frac{w}{\prod_{j=1}^n\sigma_j}\right)^{2/n}\right)\,dw
<\infty.
\]
Therefore $\int_x^\infty A_s(w)e^{-S_s(w)}\,dw$ converges (hence absolutely) for all sufficiently large $x$,
which verifies hypothesis \textup{(iv)} in our application.

\section{General n-case: setup and tail integral}

Let $X_i\sim N(\mu_i,\sigma_i^2)$ be independent, with $\sigma_i>0$, and define
\[
Z:=\prod_{i=1}^n X_i,\qquad \bF_n(x):=\mathbb{P}(Z> x),\qquad x\to+\infty.
\]

The joint density is
\[
\prod_{i=1}^n f_{X_i}(u_i)
=
\frac{1}{(2\pi)^{n/2}\prod_{i=1}^n\sigma_i}\,
\exp\!\left\{-\sum_{i=1}^n\frac{(u_i-\mu_i)^2}{2\sigma_i^2}\right\}.
\]
Expand the exponent:
\[
\sum_{i=1}^n\frac{(u_i-\mu_i)^2}{2\sigma_i^2}
=
\underbrace{\sum_{i=1}^n\frac{\mu_i^2}{2\sigma_i^2}}_{-\log C}
+
\underbrace{\sum_{i=1}^n\left(\frac{u_i^2}{2\sigma_i^2}-\frac{\mu_i}{\sigma_i^2}u_i\right)}_{=:\ \Psi(\bu)},
\qquad \bu=(u_1,\dots,u_n).
\]
Thus
\[
C:=\exp\!\left(-\sum_{i=1}^n\frac{\mu_i^2}{2\sigma_i^2}\right),
\qquad
\Psi(\bu):=\sum_{i=1}^n\left(\frac{u_i^2}{2\sigma_i^2}-\frac{\mu_i}{\sigma_i^2}u_i\right),
\]
and
\begin{equation}
\bF_n(x)
=
\frac{C}{(2\pi)^{n/2}\prod_{i=1}^n\sigma_i}
\int_{\{\prod_{i=1}^n u_i\ge x\}}
\exp\!\big(-\Psi(\bu)\big)\,\dd \bu.
\label{eq:Fn_integral}
\end{equation}

\section{Geometry of the constraint: which sign patterns matter}

We are in the right tail $x>0$. The constraint $\prod_{i=1}^n u_i\ge x>0$ forces
$\prod_{i=1}^n u_i>0$, i.e. an even number of negative coordinates.

Equivalently, introduce a sign vector $s=(s_1,\dots,s_n)\in\{\pm1\}^n$ and call a sign pattern
admissible if
\begin{equation}
\prod_{i=1}^n s_i=+1.
\label{eq:admissible_signs}
\end{equation}

There are $2^{n-1}$ admissible sign patterns.

\section{Regime decomposition: the unbalanced regions are negligible}

If at least one coordinate is too small, the product constraint forces
some other coordinate to be huge, and Gaussian tails kill that contribution.

Define
\[
a_x:=\frac{x^{1/n}}{\log x},
\qquad
b_x:=\left(\frac{x}{a_x}\right)^{1/(n-1)}
= x^{1/n}(\log x)^{1/(n-1)}.
\]

\paragraph{Balanced vs unbalanced sets.}
Define
\[
R_2(x):=\Big\{\bu:\ \prod_{j=1}^n u_j\ge x,\ \exists\, i\in\{1,\dots,n\}\ \text{s.t.}\ |u_i|\le a_x\Big\},
\]
\[
R_1(x):=\Big\{\bu:\ \prod_{j=1}^n u_j\ge x,\ \forall\, i\in\{1,\dots,n\}\ \text{we have}\ |u_i|> a_x\Big\}.
\]
Clearly
\[
R_1(x)\cap R_2(x)=\varnothing,
\qquad
\Big\{\bu:\ \prod_{j=1}^n u_j\ge x\Big\}=R_1(x)\cup R_2(x).
\]

Write
\[
I_1(x):=\frac{C}{(2\pi)^{n/2}\prod_{i=1}^n\sigma_i}\int_{R_1(x)}e^{-\Psi(\bu)}\,\dd\bu,
\qquad
I_2(x):=\frac{C}{(2\pi)^{n/2}\prod_{i=1}^n\sigma_i}\int_{R_2(x)}e^{-\Psi(\bu)}\,\dd\bu,
\]
so $\bF_n(x)=I_1(x)+I_2(x)$.

\paragraph{Step 1: upper bound for $I_2(x)$.}
If $\bu\in R_2(x)$ then $\prod_j u_j\ge x>0$, hence $\prod_j |u_j|\ge x$, and there exists
$i$ with $|u_i|\le a_x$.
If we also had $\max_j|u_j|<b_x$, then
\[
\prod_{j=1}^n |u_j|
< a_x\,(b_x)^{n-1}=x,
\]
because $\max_j|u_j|<b_x$, a contradiction. Hence on $R_2(x)$,
\[
\max_{1\le j\le n}|u_j|\ge b_x.
\]
Therefore
\[
I_2(x)=\mathbb{P}(X\in R_2(x))
\le \mathbb{P}\!\left(\max_{1\le j\le n}|X_j|\ge b_x\right)
= \mathbb{P}\!\left(\bigcup_{j=1}^n \{|X_j|\ge b_x\}\right)
\le \sum_{j=1}^n \mathbb{P}(|X_j|\ge b_x).
\]
Let us temporarily fix $j\in\{1,\dots,n\}$ and let $Y\sim N(\mu_j,\sigma_j^2)$.
By Markov's inequality, for any $t>0$ and any $u\in\mathbb{R}$,
\[
\mathbb{P}(Y\ge u)
=
\mathbb{P}(e^{tY}\ge e^{tu})
\le e^{-tu}\,\mathbb{E}e^{tY}
=
\exp\!\left(-tu+\mu_j t+\frac{\sigma_j^2 t^2}{2}\right).
\]
If $u>\mu_j$, the right-hand side is minimized at
\[
t=\frac{u-\mu_j}{\sigma_j^2}>0,
\]
and therefore
\[
\mathbb{P}(Y\ge u)\le
\exp\!\left(-\frac{(u-\mu_j)^2}{2\sigma_j^2}\right).
\]
Applying the same argument to $-Y\sim N(-\mu_j,\sigma_j^2)$, for $u\ge -\mu_j$ we get
\[
\mathbb{P}(Y\le -u)\le
\exp\!\left(-\frac{(u+\mu_j)^2}{2\sigma_j^2}\right).
\]
Hence, for $u\ge |\mu_j|$,
\[
\mathbb{P}(|X_j|\ge u)
=
\mathbb{P}(X_j\ge u)+\mathbb{P}(X_j\le -u)
\le
\exp\!\left(-\frac{(u-\mu_j)^2}{2\sigma_j^2}\right)
+
\exp\!\left(-\frac{(u+\mu_j)^2}{2\sigma_j^2}\right).
\]
Applying this bound with $u=b_x$ (which holds for all sufficiently large $x$), we obtain
\[
\mathbb{P}(|X_j|\ge b_x)
\le
\exp\!\left(-\frac{(b_x-\mu_j)^2}{2\sigma_j^2}\right)
+
\exp\!\left(-\frac{(b_x+\mu_j)^2}{2\sigma_j^2}\right).
\]
Hence,
\[
I_2(x)
\le
\sum_{j=1}^n \Bigg[
\exp\!\Big(-\frac{(b_x-\mu_j)^2}{2\sigma_j^2}\Big)
+
\exp\!\Big(-\frac{(b_x+\mu_j)^2}{2\sigma_j^2}\Big)
\Bigg],
\qquad
b_x=\left(\frac{x}{a_x}\right)^{1/(n-1)},
\qquad
a_x=\frac{x^{1/n}}{\log x}.
\]

\paragraph{Step 2: lower bound for $I_1(x)$.}
Define the balanced point in the region where all coordinates are positive:
\[
u_{0,i}(x):=\sigma_i\left(\frac{x}{\prod_{k=1}^n\sigma_k}\right)^{1/n},
\qquad i=1,\dots,n,
\]
so that $\prod_{i=1}^n u_{0,i}(x)=x$.
Set
\[
\delta_x:=\frac{1}{2n}\left(\frac{\prod_{k=1}^n\sigma_k}{x}\right)^{(n-1)/n},
\qquad
S_+(x):=\prod_{i=1}^n\Big(u_{0,i}(x),\,u_{0,i}(x)+\sigma_i\delta_x\Big].
\]
Then $S_+(x)\subset\{\prod u_i> x\}$.

Thus
\[
I_1(x)\ge \frac{C}{(2\pi)^{n/2}\prod_{i=1}^n\sigma_i}\int_{S_+(x)}e^{-\Psi(\bu)}\,\dd\bu.
\]

\paragraph{Step 2a: control $\Psi$ on $S_+(x)$.}
For $\bu\in S_+(x)$ we have
\[
0\le u_i-u_{0,i}(x)\le \sigma_i\delta_x,
\qquad i=1,\dots,n.
\]
Using $\Psi(\bu)=\sum_{i=1}^n\left(\frac{u_i^2}{2\sigma_i^2}-\frac{\mu_i}{\sigma_i^2}u_i\right)$, expand:
\[
\Psi(\bu)-\Psi(\bu_0(x))
=
\sum_{i=1}^n\left[
\frac{(u_i-u_{0,i}(x))^2}{2\sigma_i^2}
+\frac{(u_i-u_{0,i}(x))(u_{0,i}(x)-\mu_i)}{\sigma_i^2}
\right].
\]
Since $|u_i-u_{0,i}(x)|\le \sigma_i\delta_x$ and $u_{0,i}(x)=\sigma_i r(x)$, each bracket is bounded by
\[
\frac{\delta_x^2}{2}+\delta_x\left(r(x)+\frac{|\mu_i|}{\sigma_i}\right).
\]
Moreover,
\[
\delta_x r(x)=\frac{1}{2n}r(x)^{-(n-2)}\le \frac14 \qquad (x\ \text{large}),
\]
(where the value $\frac14$ comes from plugging $n=2$). Since $\delta_x\to0$, the bound
\[
\frac{\delta_x^2}{2}+\delta_x\left(r(x)+\frac{|\mu_i|}{\sigma_i}\right)
\]
is below some constant for large $x$. Hence there exists $K$ (independent of $x$) such that
\[
\Psi(\bu)\le \Psi(\bu_0(x))+K,
\qquad \bu\in S_+(x).
\]
and therefore
\[
e^{-\Psi(\bu)}\ge e^{-K}e^{-\Psi(\bu_0(x))},
\qquad \bu\in S_+(x).
\]

\paragraph{Step 2b: integrate the lower bound.}
Here $|S_+(x)|$ denotes the volume of $S_+(x)$, and
\[
|S_+(x)|
=\prod_{i=1}^n\Big(u_{0,i}(x)+\sigma_i\delta_x-u_{0,i}(x)\Big)
=\prod_{i=1}^n(\sigma_i\delta_x).
\]
Therefore

\begin{align*}
I_1(x)&\ge \frac{C}{(2\pi)^{n/2}\prod_{i=1}^n\sigma_i}\int_{S_+(x)}e^{-\Psi(\bu)}\,\dd\bu \\
&\ge
\frac{C}{(2\pi)^{n/2}\prod_{i=1}^n\sigma_i}\,e^{-K}e^{-\Psi(\bu_0(x))}\,|S_+(x)| \\
&=
\frac{C}{(2\pi)^{n/2}}\,e^{-K}e^{-\Psi(\bu_0(x))}\,\delta_x^{\,n}.
\end{align*}

\paragraph{Conclusion.}
From Step 1 and Step 2b, for all large $x$,
\[
0\le \frac{I_2(x)}{I_1(x)}
\le
\frac{\displaystyle \sum_{j=1}^n\left[
\exp\!\Big(-\frac{(b_x-\mu_j)^2}{2\sigma_j^2}\Big)
+
\exp\!\Big(-\frac{(b_x+\mu_j)^2}{2\sigma_j^2}\Big)
\right]}
{\displaystyle \frac{C}{(2\pi)^{n/2}}\,e^{-K}e^{-\Psi(\bu_0(x))}\,\delta_x^{\,n}}
\;\longrightarrow\;0,
\qquad x\to\infty.
\]
For the remainder of the paper we may restrict attention to the balanced region $R_1(x)$.

\section{Saddle system}
\label{sec:saddle}

Fix an admissible sign pattern $s\in\{\pm1\}^n$ with $\prod s_i=+1$.
Minimize $\Psi(\bu)$ on the boundary
\[
M_x^{(s)}:=\Big\{\bu:\ \prod_{i=1}^n u_i=x,\ \ s_i u_i>0\ \forall i\Big\}.
\]

Introduce the Lagrangian
\[
\mathcal{L}(\bu,\lambda)=\Psi(\bu)+\lambda\Big(\prod_{i=1}^n u_i-x\Big).
\]
Stationarity gives for each $i$:
\[
\partial_{u_i}\Psi(\bu)+\lambda\prod_{j\ne i}u_j=0
\quad\Longleftrightarrow\quad
\frac{u_i-\mu_i}{\sigma_i^2}+\lambda\frac{x}{u_i}=0
\]
\paragraph{Reduction via the common scalar $\lambda x$.}
From the stationarity conditions,
\[
\frac{u_i-\mu_i}{\sigma_i^2}+\lambda\prod_{j\ne i}u_j=0,
\qquad i=1,\dots,n,
\]
and using the constraint $\prod_{j=1}^n u_j=x$ (so $\prod_{j\ne i}u_j=x/u_i$), we get
\[
\frac{u_i-\mu_i}{\sigma_i^2}+\lambda\frac{x}{u_i}=0
\quad\Longleftrightarrow\quad
u_i^2-\mu_i u_i+\sigma_i^2(\lambda x)=0.
\]
Thus every coordinate shares the same scalar $\lambda x$, i.e.
\[
\lambda x=\frac{\mu_i u_i-u_i^2}{\sigma_i^2},\qquad i=1,\dots,n.
\]
Equating the expressions for indices $i$ and $k$ eliminates $\lambda x$ and yields the pairwise identity
\begin{equation}
\sigma_k^2(\mu_i u_i-u_i^2)=\sigma_i^2(\mu_k u_k-u_k^2),
\qquad 1\le i,k\le n.
\label{eq:lambda_common_pairwise}
\end{equation}

\paragraph{Uniqueness of the stationary point in $M_x^{(s)}$.}
Set $\beta:=-\lambda x$. For the boundary saddle one has $\lambda x<0$, hence $\beta>0$.
Then the equations are equivalent to
\[
u_i^2-\mu_i u_i-\sigma_i^2\beta=0,\qquad i=1,\dots,n,
\]
together with the sign constraints $s_i u_i>0$ and the product constraint $\prod_{i=1}^n u_i=x$.
For each $i$ the sign constraint selects a unique root, namely
\[
u_i(\beta)=\frac{\mu_i\pm\sqrt{\mu_i^2+4\sigma_i^2\beta}}{2},
\qquad i=1,\dots,n.
\]
Define $g_s(\beta):=\prod_{i=1}^n u_i(\beta)$. Then $g_s$ is strictly increasing on $(0,\infty)$ since
\[
\frac{d}{d\beta}\log|u_i(\beta)|
=\frac{u_i'(\beta)}{u_i(\beta)}
=\frac{\pm\,\sigma_i^2}{u_i(\beta)\sqrt{\mu_i^2+4\sigma_i^2\beta}}>0,
\]
where the sign $\pm$ is the same one used in the definition
\[
u_i(\beta)=\frac{\mu_i\pm\sqrt{\mu_i^2+4\sigma_i^2\beta}}{2},
\]
so $u_i(\beta)$ has the same sign as $\pm$ and hence $(\pm)/u_i(\beta)>0$.

Therefore $g_s'(\beta)>0$.
Hence the equation $g_s(\beta)=x$ has at most one solution, so the stationarity system has at most one solution in $M_x^{(s)}$ for each fixed $s$.

\paragraph{Leading order}
Rewrite \eqref{eq:lambda_common_pairwise} as
\[
\sigma_k^2 u_i^2-\sigma_i^2 u_k^2
=
\sigma_k^2\mu_i u_i-\sigma_i^2\mu_k u_k.
\]
At the boundary saddle we have $|u_i|\to\infty$ as $x\to\infty$, so the quadratic terms dominate the linear ones.
Keeping only the dominant terms gives
\[
\sigma_k^2 u_i^2 \approx \sigma_i^2 u_k^2,
\qquad 1\le i,k\le n.
\]
Equivalently,
\[
\frac{u_i^2}{\sigma_i^2}\approx \frac{u_k^2}{\sigma_k^2},
\qquad 1\le i,k\le n,
\]
so all ratios $u_i^2/\sigma_i^2$ are (asymptotically) the same. Hence there exists a scalar $r(x)>0$
such that, for each $i$,
\[
u_i^2 \approx \sigma_i^2\,r(x)^2,
\qquad\text{that is}\qquad
u_i \approx s_i\,\sigma_i\,r(x),
\qquad s_i\in\{\pm1\}.
\]
Imposing the constraint $\prod_{i=1}^n u_i=x$ then yields
\[
x=\prod_{i=1}^n u_i \approx \Big(\prod_{i=1}^n s_i\Big)\Big(\prod_{i=1}^n \sigma_i\Big)\,r(x)^n.
\]
Since $x>0$, admissible sign patterns satisfy $\prod_{i=1}^n s_i=+1$, hence
\[
r(x)^n \approx \frac{x}{\prod_{i=1}^n \sigma_i}
\quad\Longrightarrow\quad
r(x) \approx \left(\frac{x}{\prod_{i=1}^n \sigma_i}\right)^{1/n}.
\]
Therefore the leading-order saddle shape is the balanced scale
\begin{equation}
u_i \approx s_i\,\sigma_i\left(\frac{x}{\prod_{j=1}^n \sigma_j}\right)^{1/n},
\qquad s_i\in\{\pm1\},\qquad \prod_{i=1}^n s_i=+1.
\label{eq:leading_shape}
\end{equation}

\section{Constant-order correction at a fixed sign pattern}

Recall the leading balanced size
\begin{equation}
r(x):=\left(\frac{x}{\prod_{j=1}^n \sigma_j}\right)^{1/n},
\qquad r(x)\to\infty \ \ (x\to\infty).
\label{eq:r_def}
\end{equation}
Fix an admissible sign pattern $s\in\{\pm1\}^n$ with $\prod_i s_i=+1$, and set
\[
\Sigma:=\prod_{i=1}^n \sigma_i
\qquad\text{so that}\qquad
\Sigma\,r(x)^n=x.
\]

\paragraph{Step 1: plug $u_i=s_i\sigma_i r+\delta_i$ into the pairwise identity.}
From Section~\ref{sec:saddle} (eliminating $\lambda$), for every pair $i,k$,
\begin{equation}
\frac{\mu_i u_i-u_i^2}{\sigma_i^2}=\frac{\mu_k u_k-u_k^2}{\sigma_k^2}.
\label{eq:lambda_common_pairwise_n}
\end{equation}
Equivalently,
\begin{equation}
\sigma_k^2 u_i^2-\sigma_i^2 u_k^2=\sigma_k^2\mu_i u_i-\sigma_i^2\mu_k u_k .
\label{eq:pairwise_square_linear}
\end{equation}

Write
\begin{equation}
u_i=s_i\sigma_i r+\delta_i,
\qquad
u_k=s_k\sigma_k r+\delta_k,
\qquad r\to\infty,
\label{eq:ui_delta_ansatz}
\end{equation}
with $\delta_i,\delta_k=O(1)$. Expanding $u_i^2,u_k^2$ and inserting into
\eqref{eq:pairwise_square_linear}, the $r^2$ terms cancel, and matching the remaining $r$--terms yields
\[
\sigma_k\, s_i\,(2\delta_i-\mu_i)=\sigma_i\, s_k\,(2\delta_k-\mu_k).
\]
Hence $\dfrac{s_i}{\sigma_i}(2\delta_i-\mu_i)$ is the same for all $i$, so there exists a constant $b_s$ such that
\begin{equation}
\frac{s_i}{\sigma_i}\,(2\delta_i-\mu_i)=b_s,
\qquad i=1,\dots,n,
\label{eq:delta_common_constant}
\end{equation}
i.e.
\begin{equation}
\delta_i=\frac{\mu_i}{2}+\frac{b_s}{2}\,s_i\sigma_i,
\qquad i=1,\dots,n.
\label{eq:delta_solution}
\end{equation}

\paragraph{Step 1b: determine $b_s$ from the product constraint.}
Factor $s_i\sigma_i r$ and set
\begin{equation}
a_i:=\frac{\delta_i}{s_i\sigma_i}.
\label{eq:ai_def}
\end{equation}


\paragraph{Product expansion.}
Write
\[
u_i=s_i\sigma_i r\left(1+\frac{a_i}{r}\right),
\qquad
P(r):=\prod_{i=1}^n\left(1+\frac{a_i}{r}\right),
\qquad
S_1:=\sum_{i=1}^n a_i,\ \ S_2:=\sum_{i=1}^n a_i^2 .
\]
Since $a_i=O(1)$, we have $a_i/r\to0$, so using $\log(1+t)=t-\tfrac12 t^2+O(t^3)$ gives
\[
\log P(r)
=\sum_{i=1}^n\log\!\left(1+\frac{a_i}{r}\right)
=\frac{S_1}{r}-\frac{S_2}{2r^2}+O(r^{-3}).
\]
Let $z:=\frac{S_1}{r}-\frac{S_2}{2r^2}+O(r^{-3})$. Then $z=O(r^{-1})$, hence
$e^{z}=1+z+\tfrac12 z^2+O(z^3)$ and $z^2=\frac{S_1^2}{r^2}+O(r^{-3})$, $z^3=O(r^{-3})$.
Therefore
\[
P(r)
=1+\frac{S_1}{r}+\frac{S_1^2-S_2}{2r^2}+O(r^{-3}),
\]
and consequently
\begin{equation}
\prod_{i=1}^n u_i
=\Big(\prod_{i=1}^n s_i\sigma_i\Big)\,r^n
\left(
1+\frac{1}{r}\sum_{i=1}^n a_i+\frac{1}{2r^2}\Big[\Big(\sum_{i=1}^n a_i\Big)^2-\sum_{i=1}^n a_i^2\Big]
+O(r^{-3})
\right).
\label{eq:prod_up_to_rminus3}
\end{equation}

Then
\[
u_i=s_i\sigma_i r\left(1+\frac{a_i}{r}\right),
\qquad
\prod_{i=1}^n u_i
=\Big(\prod_{i=1}^n s_i\sigma_i\Big)\,r^n
\left(1+\frac{1}{r}\sum_{i=1}^n a_i+O(r^{-2})\right).
\]
Since $\prod_i s_i=+1$, we have $\prod_i s_i\sigma_i=\Sigma$, hence
\begin{equation}
\prod_{i=1}^n u_i
=
\Sigma\,r^n\left(1+\frac{1}{r}\sum_{i=1}^n a_i+O(r^{-2})\right).
\label{eq:prod_first}
\end{equation}
Using $\Sigma r^n=x$, \eqref{eq:prod_first} becomes
\[
\prod_{i=1}^n u_i
=
x\left(1+\frac{1}{r}\sum_{i=1}^n a_i+O(r^{-2})\right).
\]
Imposing the exact constraint $\prod_i u_i=x$ forces the $1/r$ coefficient to vanish:
\begin{equation}
\sum_{i=1}^n a_i=0.
\label{eq:sum_ai_zero}
\end{equation}
Using \eqref{eq:delta_solution},
\[
a_i=\frac{\delta_i}{s_i\sigma_i}
=\frac{\mu_i}{2s_i\sigma_i}+\frac{b_s}{2}.
\]
Thus \eqref{eq:sum_ai_zero} gives
\begin{equation}
0=\sum_{i=1}^n\left(\frac{\mu_i}{2s_i\sigma_i}+\frac{b_s}{2}\right)
\quad\Longrightarrow\quad
b_s=-\frac{1}{n}\sum_{i=1}^n\frac{\mu_i}{s_i\sigma_i}.
\label{eq:bs_solution_n}
\end{equation}

\paragraph{Step 2: one order further.}
Refine the expansion by allowing a coordinate-dependent $O(r^{-1})$ correction:
\begin{equation}
u_i=s_i\sigma_i r+\delta_i+\frac{\varepsilon_i}{r}+O(r^{-2}),
\qquad i=1,\dots,n.
\label{eq:ui_eps_ansatz}
\end{equation}

Since the quantity
\[
\frac{\mu_i u_i-u_i^2}{\sigma_i^2}
\]
is the same for all $i$ by \eqref{eq:lambda_common_pairwise_n}, substituting
\eqref{eq:ui_eps_ansatz} and using \eqref{eq:delta_solution} give
\[
\frac{\mu_i\delta_i-\delta_i^2-2s_i\sigma_i\varepsilon_i}{\sigma_i^2}
=
\frac{\mu_k\delta_k-\delta_k^2-2s_k\sigma_k\varepsilon_k}{\sigma_k^2},
\qquad 1\le i,k\le n.
\]
Hence there exists a constant $d_s$ such that
\begin{equation}
\frac{\mu_i\delta_i-\delta_i^2-2s_i\sigma_i\varepsilon_i}{\sigma_i^2}=d_s,
\qquad i=1,\dots,n,
\label{eq:eps_common_constant}
\end{equation}
that is,
\begin{equation}
\varepsilon_i
=
\frac{\mu_i\delta_i-\delta_i^2-d_s\sigma_i^2}{2s_i\sigma_i},
\qquad i=1,\dots,n.
\label{eq:eps_solution_pre}
\end{equation}

\paragraph{Product expansion to the next order.}
Write
\[
u_i=s_i\sigma_i r\left(1+\frac{a_i}{r}+\frac{\rho_i}{r^2}+O(r^{-3})\right),
\qquad
a_i:=\frac{\delta_i}{s_i\sigma_i},
\qquad
\rho_i:=\frac{\varepsilon_i}{s_i\sigma_i}.
\]
Let
\[
P(r):=\prod_{i=1}^n\left(1+\frac{a_i}{r}+\frac{\rho_i}{r^2}+O(r^{-3})\right),
\qquad
S_1:=\sum_{i=1}^n a_i,\qquad
S_2:=\sum_{i=1}^n a_i^2.
\]
Since $a_i=O(1)$ and $\rho_i=O(1)$, we have
\[
\frac{a_i}{r}+\frac{\rho_i}{r^2}+O(r^{-3})\to0,
\]
so using $\log(1+t)=t-\tfrac12 t^2+O(t^3)$ gives
\[
\log P(r)
=
\sum_{i=1}^n
\log\!\left(1+\frac{a_i}{r}+\frac{\rho_i}{r^2}+O(r^{-3})\right).
\]
Now
\[
\log\!\left(1+\frac{a_i}{r}+\frac{\rho_i}{r^2}+O(r^{-3})\right)
=
\frac{a_i}{r}
+\frac{\rho_i}{r^2}
-\frac{a_i^2}{2r^2}
+O(r^{-3}),
\]
hence
\[
\log P(r)
=
\frac{S_1}{r}
+\frac{1}{r^2}\left(\sum_{i=1}^n \rho_i-\frac12 S_2\right)
+O(r^{-3}).
\]
Let
\[
z:=\frac{S_1}{r}
+\frac{1}{r^2}\left(\sum_{i=1}^n \rho_i-\frac12 S_2\right)
+O(r^{-3}).
\]
Then $z=O(r^{-1})$, so
\[
e^z=1+z+\frac12 z^2+O(z^3),
\qquad
z^2=\frac{S_1^2}{r^2}+O(r^{-3}),
\qquad
z^3=O(r^{-3}).
\]
Therefore
\[
P(r)
=
1+\frac{S_1}{r}
+\frac{1}{r^2}\left(
\sum_{i=1}^n \rho_i+\frac12(S_1^2-S_2)
\right)
+O(r^{-3}).
\]
Consequently,
\[
\prod_{i=1}^n u_i
=
\Sigma r^n
\left(
1+\frac{1}{r}\sum_{i=1}^n a_i
+\frac{1}{r^2}\left[\sum_{i=1}^n \rho_i+\frac12\left(\left(\sum_{i=1}^n a_i\right)^2-\sum_{i=1}^n a_i^2\right)\right]
+O(r^{-3})
\right).
\]
Since \eqref{eq:sum_ai_zero} gives $\sum_{i=1}^n a_i=0$, the exact constraint $\prod_{i=1}^n u_i=x=\Sigma r^n$
forces
\begin{equation}
\sum_{i=1}^n \rho_i=\frac12\sum_{i=1}^n a_i^2.
\label{eq:sum_rhoi_halfS2}
\end{equation}

Using \eqref{eq:eps_solution_pre} and $a_i=\delta_i/(s_i\sigma_i)$, equation
\eqref{eq:sum_rhoi_halfS2} becomes
\[
\sum_{i=1}^n\frac{\mu_i\delta_i-\delta_i^2-d_s\sigma_i^2}{2\sigma_i^2}
=
\frac12\sum_{i=1}^n\frac{\delta_i^2}{\sigma_i^2},
\]
hence
\[
\sum_{i=1}^n\frac{\mu_i\delta_i-2\delta_i^2}{\sigma_i^2}=nd_s.
\]
Now from \eqref{eq:delta_solution},
\[
\delta_i=\frac{\mu_i}{2}+\frac{b_s}{2}s_i\sigma_i,
\]
so
\[
\frac{\mu_i\delta_i-2\delta_i^2}{\sigma_i^2}
=
-\frac{b_s^2}{2}-\frac{b_s}{2}\frac{\mu_i}{s_i\sigma_i}.
\]
Summing over $i$ and using \eqref{eq:bs_solution_n},
\[
\sum_{i=1}^n\frac{\mu_i\delta_i-2\delta_i^2}{\sigma_i^2}
=
-\frac{n b_s^2}{2}
-\frac{b_s}{2}\sum_{i=1}^n\frac{\mu_i}{s_i\sigma_i}
=
-\frac{n b_s^2}{2}+\frac{n b_s^2}{2}=0.
\]
Therefore
\[
d_s=0.
\]

Substituting $d_s=0$ into \eqref{eq:eps_solution_pre} gives
\[
\varepsilon_i
=
\frac{\mu_i\delta_i-\delta_i^2}{2s_i\sigma_i}
=
\frac{s_i}{8}\left(\frac{\mu_i^2}{\sigma_i}-b_s^2\sigma_i\right).
\]

\paragraph{Conclusion.}
\[
u_i
=
s_i\sigma_i r+\delta_i+\frac{\varepsilon_i}{r}+O(r^{-2}),
\qquad i=1,\dots,n,
\]
that is,
\begin{align*}
u_i
&=
s_i\sigma_i r
+\frac12\left(\mu_i+\;b_s\,s_i\sigma_i\right)
+\frac{s_i}{8r}\left(\frac{\mu_i^2}{\sigma_i}-b_s^2\sigma_i\right)
+O(r^{-2}) \\
&=
s_i\sigma_i r
+\frac12\left(
\mu_i-\frac{s_i\sigma_i}{n}\sum_{j=1}^n\frac{\mu_j}{s_j\sigma_j}
\right)
+\frac{s_i}{8r}\left(
\frac{\mu_i^2}{\sigma_i}
-\frac{\sigma_i}{n^2}\left(\sum_{j=1}^n\frac{\mu_j}{s_j\sigma_j}\right)^2
\right)
+O(r^{-2}),
\qquad i=1,\dots,n.
\end{align*}

Equivalently,
\begin{align*}
u_i
=
s_i\sigma_i r\Bigg(
1
&+\frac{1}{2r}\left(
\frac{\mu_i}{s_i\sigma_i}
-\frac{1}{n}\sum_{j=1}^n\frac{\mu_j}{s_j\sigma_j}
\right) \\
&+\frac{1}{8r^2}\left(
\left(\frac{\mu_i}{\sigma_i}\right)^2
-\frac{1}{n^2}\left(\sum_{j=1}^n\frac{\mu_j}{s_j\sigma_j}\right)^2
\right)
+O(r^{-3})
\Bigg),
\qquad i=1,\dots,n.
\end{align*}

\section{Dominating exponent at the saddle and removing the sign-pattern sum}

Fix an admissible sign pattern $s\in\{\pm1\}^n$ with $\prod_{i=1}^n s_i=+1$.
Let $\bu^{(s)}(x)=(u_{1,s}(x),\dots,u_{n,s}(x))$.

Recall
\[
\Psi(\bu)=\sum_{i=1}^n\left(\frac{u_i^2}{2\sigma_i^2}-\frac{\mu_i}{\sigma_i^2}u_i\right),
\qquad
r(x)=\left(\frac{x}{\prod_{j=1}^n\sigma_j}\right)^{1/n}\to\infty.
\]

\paragraph{Exponent at the saddle.}
Define, for each admissible $s$,
\[
L_s:=\sum_{i=1}^n s_i\frac{\mu_i}{\sigma_i}.
\]
Then,
\begin{equation}
\begin{aligned}
\Psi\!\big(\bu^{(s)}(x)\big)
={}&
\frac{n}{2}\,r(x)^2-r(x)\,L_s
-\frac{1}{4}\left(
\sum_{i=1}^n\left(\frac{\mu_i}{\sigma_i}\right)^2
-\frac{1}{n}L_s^2
\right) \\
&\quad
+\frac{1}{16\,r(x)}
\left(
\frac{2}{n^2}L_s^3
-\frac{L_s}{n}\sum_{i=1}^n\left(\frac{\mu_i}{\sigma_i}\right)^2
-\sum_{i=1}^n s_i\left(\frac{\mu_i}{\sigma_i}\right)^3
\right)
+O\!\big(r(x)^{-2}\big),
\qquad x\to\infty.
\end{aligned}
\label{eq:Psi_saddle_n_recall}
\end{equation}
Hence
\begin{equation}
\begin{aligned}
\exp\!\big(-\Psi(\bu^{(s)}(x))\big)
={}&
\exp\!\Bigg\{
-\frac{n}{2}\,r(x)^2+r(x)\,L_s
+\frac{1}{4}\left(
\sum_{i=1}^n\left(\frac{\mu_i}{\sigma_i}\right)^2
-\frac{1}{n}L_s^2
\right) \\
&\quad
+\frac{1}{16\,r(x)}
\left(
\sum_{i=1}^n s_i\left(\frac{\mu_i}{\sigma_i}\right)^3
+\frac{L_s}{n}\sum_{i=1}^n\left(\frac{\mu_i}{\sigma_i}\right)^2
-\frac{2}{n^2}L_s^3
\right)
\Bigg\}
\Big(1+O(r(x)^{-2})\Big).
\end{aligned}
\label{eq:exp_saddle_n_recall}
\end{equation}
\paragraph{Admissible sign patterns.}
Let
\[
\mathcal{S}:=\Big\{s\in\{\pm1\}^n:\ \prod_{i=1}^n s_i=+1\Big\},
\qquad
|\mathcal{S}|=2^{n-1}.
\]

\section{Prefactor}

We apply a two-step Laplace scheme summarized in Section~\ref{sec:tools_lemmas}.

\paragraph{Step 1: change variables $(u_1,\dots,u_{n-1},w)$.}
Let
\[
w:=\prod_{i=1}^n u_i,
\qquad
(u_1,\dots,u_{n-1})=(u_1,\dots,u_{n-1}),
\qquad
u_n=\frac{w}{u_1\cdots u_{n-1}}.
\]
Then the Jacobian is
\[
\dd u_1\,\dd u_2\cdots \dd u_n
=
\frac{1}{|u_1\cdots u_{n-1}|}\,
\dd u_1\,\dd u_2\cdots \dd u_{n-1}\,\dd w.
\]

and $\{\prod u_i> x\}$ becomes $\{w> x\}$. Hence \eqref{eq:Fn_integral} becomes
\begin{equation}
\bF_n(x)
=
\frac{C}{(2\pi)^{n/2}\prod\sigma_i}
\int_{w=x}^\infty
\left[
\int_{\R^{n-1}}
\exp\!\Big(-\Phi_w(\tilde{\bu})\Big)\,
\frac{\dd \tilde{\bu}}{|u_1\cdots u_{n-1}|}
\right]\dd w,
\label{eq:Fn_uw}
\end{equation}
where $\tilde{\bu}=(u_1,\dots,u_{n-1})$ and
\[
\Phi_w(\tilde{\bu})
:=
\Psi\!\left(u_1,\dots,u_{n-1},\frac{w}{u_1\cdots u_{n-1}}\right).
\]

\paragraph{Step 2: Laplace in $(u_1,\dots,u_{n-1})$ at fixed $w$.}
Recall
\[
u_n=\frac{w}{u_1\cdots u_{n-1}},
\qquad
\Phi_w(u_1,\dots,u_{n-1})=\Psi\!\left(u_1,\dots,u_{n-1},\frac{w}{u_1\cdots u_{n-1}}\right).
\]
Fix an admissible sign region $s$ and let the minimizer be
\[
(u_{1,s}(w),\dots,u_{n-1,s}(w)),
\qquad
u_{n,s}(w):=\frac{w}{u_{1,s}(w)\cdots u_{n-1,s}(w)},
\qquad
S_s(w):=\Phi_w(u_{1,s}(w),\dots,u_{n-1,s}(w)).
\]

\emph{(a) Jacobian factor at the minimizer.}
From Section~8 with $x$ replaced by $w$,
\[
u_{i,s}(w)
=
s_i\sigma_i r(w)\left(
1+\frac{1}{2r(w)}\left(s_i\frac{\mu_i}{\sigma_i}-\frac{L_s}{n}\right)
+O(r(w)^{-2})
\right),
\qquad i=1,\dots,n.
\]
Therefore
\[
\prod_{i=1}^{n-1}u_{i,s}(w)
=
\Big(\prod_{i=1}^{n-1}s_i\sigma_i\Big)\,r(w)^{n-1}
\left(
1+\frac{1}{2r(w)}\sum_{i=1}^{n-1}\left(s_i\frac{\mu_i}{\sigma_i}-\frac{L_s}{n}\right)
+O(r(w)^{-2})
\right).
\]
Since
\[
\sum_{i=1}^{n-1}\left(s_i\frac{\mu_i}{\sigma_i}-\frac{L_s}{n}\right)
=
\frac{L_s}{n}-s_n\frac{\mu_n}{\sigma_n},
\]
we get, after taking absolute values,
\[
|u_{1,s}(w)\cdots u_{n-1,s}(w)|
=
\Big(\prod_{i=1}^{n-1}\sigma_i\Big)\,r(w)^{n-1}
\left(
1+\frac{1}{2r(w)}\left(\frac{L_s}{n}-s_n\frac{\mu_n}{\sigma_n}\right)
+O(r(w)^{-2})
\right).
\]
Consequently,
\[
\frac{1}{|u_{1,s}(w)\cdots u_{n-1,s}(w)|}
=
\frac{1}{\big(\prod_{i=1}^{n-1}\sigma_i\big)\,r(w)^{n-1}}
\left(
1+\frac{1}{2r(w)}\left(s_n\frac{\mu_n}{\sigma_n}-\frac{L_s}{n}\right)
+O(r(w)^{-2})
\right).
\]

\emph{(b) Hessian determinant at the minimizer.}
Let $H_s(w)$ be the Hessian matrix in the variables $u_1,\dots,u_{n-1}$:
\[
(H_s(w))_{ij}
:=
\frac{\partial^2\Phi_w}{\partial u_i\,\partial u_j}
\Big(u_{1,s}(w),\dots,u_{n-1,s}(w)\Big).
\]

Recall that for fixed $w$ we reduce to $(n-1)$ free variables by setting
\[
u_n=\frac{w}{u_1\cdots u_{n-1}},
\qquad
\Phi_w(u_1,\dots,u_{n-1})
:=\Psi\!\left(u_1,\dots,u_{n-1},\frac{w}{u_1\cdots u_{n-1}}\right).
\]
Equivalently, $\Phi_w$ can be written out as
\[
\Phi_w(u_1,\dots,u_{n-1})
=
\sum_{k=1}^{n-1}\left(\frac{u_k^2}{2\sigma_k^2}-\frac{\mu_k}{\sigma_k^2}u_k\right)
\;+\;
\frac{1}{2\sigma_n^2}\left(\frac{w}{u_1\cdots u_{n-1}}\right)^2
\;-\;
\frac{\mu_n}{\sigma_n^2}\left(\frac{w}{u_1\cdots u_{n-1}}\right).
\]

For $i\neq j$,
\[
(H_s(w))_{ij}
=
\frac{2u_{n,s}(w)^2-\mu_n u_{n,s}(w)}
{\sigma_n^2\,u_{i,s}(w)\,u_{j,s}(w)},
\qquad
(H_s(w))_{ii}
=
\frac{1}{\sigma_i^2}
+
\frac{3u_{n,s}(w)^2-2\mu_n u_{n,s}(w)}
{\sigma_n^2\,u_{i,s}(w)^2}.
\]

Now insert the saddle expansion from Section~8. A direct expansion gives
\[
(H_s(w))_{ii}
=
\frac{4}{\sigma_i^2}\left(
1+\frac{1}{4r(w)}
\left(
s_n\frac{\mu_n}{\sigma_n}-3s_i\frac{\mu_i}{\sigma_i}
\right)
+O(r(w)^{-2})
\right),
\]
and for $i\neq j$,
\[
(H_s(w))_{ij}
=
\frac{2s_is_j}{\sigma_i\sigma_j}\left(
1+\frac{1}{2r(w)}
\left(
s_n\frac{\mu_n}{\sigma_n}
-s_i\frac{\mu_i}{\sigma_i}
-s_j\frac{\mu_j}{\sigma_j}
\right)
+O(r(w)^{-2})
\right).
\]

Equivalently,
\[
H_s(w)=H_{0,s}+\frac{1}{r(w)}H_{1,s}+O(r(w)^{-2}),
\]
where
\[
H_{0,s}=2\,D_s\,(I_{n-1}+\mathbf 1\mathbf 1^\top)\,D_s,
\qquad
D_s:=\mathrm{diag}\!\left(\frac{s_1}{\sigma_1},\dots,\frac{s_{n-1}}{\sigma_{n-1}}\right).
\]
Using
\[
\det\!\left(H_{0,s}+\frac{1}{r(w)}H_{1,s}+O(r(w)^{-2})\right)
=
\det(H_{0,s})
\left(
1+\frac{1}{r(w)}\operatorname{tr}\!\big(H_{0,s}^{-1}H_{1,s}\big)
+O(r(w)^{-2})
\right),
\]
together with
\[
\det(H_{0,s})=\frac{n\,2^{n-1}}{\prod_{i=1}^{n-1}\sigma_i^2},
\qquad
\operatorname{tr}\!\big(H_{0,s}^{-1}H_{1,s}\big)
=
s_n\frac{\mu_n}{\sigma_n}-\frac{n+1}{2n}L_s,
\]
we obtain
\[
\det H_s(w)
=
\frac{n\,2^{\,n-1}}{\prod_{i=1}^{n-1}\sigma_i^2}
\left(
1+\frac{1}{r(w)}
\left(
s_n\frac{\mu_n}{\sigma_n}-\frac{n+1}{2n}L_s
\right)
+O(r(w)^{-2})
\right).
\]
Hence
\[
\sqrt{\det H_s(w)}
=
\frac{\sqrt n\,2^{(n-1)/2}}{\prod_{i=1}^{n-1}\sigma_i}
\left(
1+\frac{1}{2r(w)}
\left(
s_n\frac{\mu_n}{\sigma_n}-\frac{n+1}{2n}L_s
\right)
+O(r(w)^{-2})
\right).
\]

Putting (a)--(b) together gives
\begin{equation}
\frac{(2\pi)^{(n-1)/2}}{|u_{1,s}(w)\cdots u_{n-1,s}(w)|\sqrt{\det H_s(w)}}
=
\frac{\pi^{(n-1)/2}}{\sqrt n\,r(w)^{n-1}}
\left(
1+\frac{n-1}{4n}\frac{L_s}{r(w)}
+O(r(w)^{-2})
\right).
\label{eq:tangent_laplace_n}
\end{equation}

\emph{(c) Laplace evaluation of the inner integral.}
Define
\[
A_s(w):=\frac{(2\pi)^{(n-1)/2}}{|u_{1,s}(w)\cdots u_{n-1,s}(w)|\sqrt{\det H_s(w)}}.
\]
Then \eqref{eq:tangent_laplace_n} gives
\[
A_s(w)
=
\frac{\pi^{(n-1)/2}}{\sqrt n\,r(w)^{n-1}}
\left(
1+\frac{n-1}{4n}\frac{L_s}{r(w)}
+O(r(w)^{-2})
\right).
\]

\paragraph{Coefficient $\kappa_s(w)$.}
From the proof of Wong's Theorem~3, the coefficients in the multidimensional Laplace
expansion are given by
\[
c_k=\sum_{|\alpha|=2k}\frac{d_\alpha}{\alpha!}\,D^\alpha G(0).
\]
where \(d_\alpha/\alpha!\) is constant. Thus it suffices to show that
\[
\frac{D^\alpha G_{s,w}(0)}{G_{s,w}(0)}=O\!\left(r(w)^{-1}\right),
\qquad |\alpha|=2.
\]
In fact, the argument below yields a sharper estimate than this, which is more than enough for our purposes.

From Wong's proof we have
\[
G_{s,w}(y)=g(h_{s,w}(y))\,\det h'_{s,w}(y),
\]
and in our application
\[
g(\tilde{\bu})=\frac{1}{|u_1\cdots u_{n-1}|},
\qquad
f(\tilde{\bu})=\frac{\Phi_w(\tilde{\bu})}{r(w)}.
\]
Also,
\[
h_{s,w}(0)=\tilde{\bu}_s(w).
\]

Now the saddle expansion gives
\[
u_{i,s}(w)=s_i\sigma_i r(w)\bigl(1+O(r(w)^{-1})\bigr),
\qquad i=1,\dots,n,
\]
so in particular
\[
u_{i,s}(w)=O(r(w))
\qquad\text{and}\qquad
\frac{1}{u_{i,s}(w)}=O(r(w)^{-1}).
\]

Next, from Wong's proof, for suitably chosen numbers \(\nu_{1,s}(w),\dots,\nu_{n-1,s}(w)\), we have
\begin{equation}
f(h_{s,w}(y))
=
f(\tilde{\bu}_s(w))
+\frac12\sum_{j=1}^{n-1}\nu_{j,s}(w)\,y_j^2.
\label{eq:hsw_normal_form}
\end{equation}

We now expand \(f\circ h_{s,w}\) in the variable \(y\) around \(y=0\). By the second-order Taylor expansion,
\begin{align}
f(h_{s,w}(y))
&=
f(h_{s,w}(0))
+
Df(h_{s,w}(0))\,Dh_{s,w}(0)\,y \notag\\
&\qquad
+
\frac12\,y^T Dh_{s,w}(0)^T D^2f(h_{s,w}(0))\,Dh_{s,w}(0)\,y
+
O\!\left(\|y\|^3\right),
\label{eq:taylor_f_hsw_y}
\end{align}

Since the right-hand side of \eqref{eq:hsw_normal_form} has no term with \(y\), we must have
\[
Df(h_{s,w}(0))\,Dh_{s,w}(0)=0.
\]
Hence,
\[
Df(h_{s,w}(0))=Df(\tilde{\bu}_s(w))=0,
\]

and therefore
\begin{align}
f(h_{s,w}(y))
&=
f(\tilde{\bu}_s(w))
+
\frac12\,y^T Dh_{s,w}(0)^T D^2f(\tilde{\bu}_s(w))\,Dh_{s,w}(0)\,y
+
O\!\left(\|y\|^3\right).
\label{eq:taylor_f_hsw_y_saddle}
\end{align}
Comparing \eqref{eq:taylor_f_hsw_y_saddle} with identity
\eqref{eq:hsw_normal_form}, we obtain
\[
Dh_{s,w}(0)^T D^2f(\tilde{\bu}_s(w))\,Dh_{s,w}(0)
=
\operatorname{diag}\!\bigl(\nu_{1,s}(w),\dots,\nu_{n-1,s}(w)\bigr).
\]
Since
\[
f(\tilde{\bu})=\frac{\Phi_w(\tilde{\bu})}{r(w)},
\]
each further derivative of \(f\) at the saddle point \(\tilde{\bu}_s(w)\) lowers the order by one power of \(r(w)^{-1}\). In particular,
\[
D^2f(\tilde{\bu}_s(w))=O\!\left(r(w)^{-1}\right),
\qquad
D^3f(\tilde{\bu}_s(w))=O\!\left(r(w)^{-2}\right),
\qquad
D^4f(\tilde{\bu}_s(w))=O\!\left(r(w)^{-3}\right).
\]
Hence, differentiating \eqref{eq:hsw_normal_form} at \(y=0\), we obtain
\[
Dh_{s,w}(0)=O(1),\qquad
D^2h_{s,w}(0)=O\!\left(r(w)^{-1}\right),\qquad
D^3h_{s,w}(0)=O\!\left(r(w)^{-2}\right).
\]

Applying the chain rule and product rule to
\[
G_{s,w}(y)=g(h_{s,w}(y))\,\det h'_{s,w}(y),
\]
we see that every second \(y\)-derivative of \(G_{s,w}\) at \(0\) gains at least one extra
factor \(r(w)^{-1}\) relative to \(G_{s,w}(0)\). Hence
\[
\frac{D^\alpha G_{s,w}(0)}{G_{s,w}(0)}=O\!\left(r(w)^{-2}\right),
\qquad |\alpha|=2.
\]
Therefore
\[
\kappa_s(w)=O\!\left(r(w)^{-2}\right),
\]
and so
\[
\frac{\kappa_s(w)}{r(w)}=O\!\left(r(w)^{-3}\right).
\]
Consequently,
\[
\int_{D_s}\exp\!\bigl(-\Phi_w(\tilde{\bu})\bigr)\,\frac{d\tilde{\bu}}{|u_1\cdots u_{n-1}|}
=
A_s(w)e^{-S_s(w)}\bigl(1+O\!\left(r(w)^{-2}\right)\bigr),
\qquad w\to\infty.
\]

\paragraph{Step 3: one--sided Laplace in $w$ at $w=x$.}
Since $(u_{1,s}(w),\dots,u_{n-1,s}(w))$ minimizes $\Phi_w$ at fixed $w$, and
\[
S_s'(w)=\frac{\partial}{\partial w}\Phi_w(u_{1,s}(w),\dots,u_{n-1,s}(w)).
\]
Only $u_n=w/(u_1\cdots u_{n-1})$ depends on $w$, hence
\[
S_s'(w)
=
\frac{\partial\Psi}{\partial u_n}(u_{1,s}(w),\dots,u_{n,s}(w))\cdot
\frac{\partial u_n}{\partial w}\Bigg|_{\text{min}}
=
\frac{u_{n,s}(w)-\mu_n}{\sigma_n^2}\cdot\frac{1}{u_{1,s}(w)\cdots u_{n-1,s}(w)}.
\]
Using $u_{1,s}(w)\cdots u_{n,s}(w)=w$ this becomes
\[
S_s'(w)=\frac{u_{n,s}(w)\big(u_{n,s}(w)-\mu_n\big)}{\sigma_n^2\,w}
=\frac{1}{\big(\prod_{j=1}^n\sigma_j\big)\,r(w)^{\,n-2}}\Big(
1-\frac{L_s}{n\,r(w)}+O\bigl(r(w)^{-2}\bigr)\Big).
\]

Thus
\[
\int_x^\infty\!\left(\int_{D_s}\exp\!\big(-\Phi_w(\tilde{\bu})\big)\frac{d\tilde{\bu}}{|u_1\cdots u_{n-1}|}\right)\!dw
=
\int_x^\infty\!A_s(w)e^{-S_s(w)}\,dw\ \Big(1+O(r(x)^{-2})\Big),
\]
since $r(w)\ge r(x)$ for $w\ge x$. Applying Corollary~\ref{cor:wong_mu1_alpha1_rate} to $A_s,S_s$ yields
\[
\int_x^\infty A_s(w)e^{-S_s(w)}\,dw
=
\frac{A_s(x)}{S_s'(x)}e^{-S_s(x)}
\left(
1+\frac{\eta_s(x)}{r(x)^2}+O(r(x)^{-4})
\right),
\]
where
\[
\eta_s(x)
=
r(x)^2\left(
\frac{A_s'(x)}{A_s(x)S_s'(x)}
-
\frac{S_s''(x)}{(S_s'(x))^2}
\right),
\qquad
r(x)=x^{1/n}\Bigl(\prod_{j=1}^n\sigma_j\Bigr)^{-1/n}.
\]
Moreover,
\[
A_s(x)=O\!\left(r(x)^{-(n-1)}\right),\qquad
S_s'(x)=O\!\left(r(x)^{-(n-2)}\right),
\]
so
\[
A_s'(x)=O\!\left(r(x)^{-2n+1}\right),\qquad
S_s''(x)=O\!\left(r(x)^{-2n+2}\right).
\]
Therefore
\[
\eta_s(x)=O(1),
\]
and hence
\[
\frac{\eta_s(x)}{r(x)^2}=O\!\left(r(x)^{-2}\right)=O\!\left(x^{-2/n}\right).
\]
In particular,
\begin{equation}
\int_x^\infty A_s(w)e^{-S_s(w)}\,dw
=
\frac{A_s(x)}{S_s'(x)}\,e^{-S_s(x)}\Big(1+O(x^{-2/n})\Big),
\label{eq:w_boundary_laplace_n}
\end{equation}
with $A_s(w)$ given by the Step~2 prefactor.

\paragraph{Stop point.}
Combining \eqref{eq:Fn_uw}, \eqref{eq:tangent_laplace_n}, and \eqref{eq:w_boundary_laplace_n} yields
\[
\bF_n(x)
=
\frac{C}{(2\pi)^{n/2}\prod_{i=1}^n \sigma_i}
\sum_{\substack{s\in\{\pm1\}^n\\ \prod s_i=+1}}
\frac{A_s(x)}{S_s'(x)}\,e^{-S_s(x)}\Big(1+O(x^{-2/n})\Big),
\]
with $S_s(x)=\Psi(\bu^{(s)}(x))$ and the dominant exponent already given in \eqref{eq:Psi_saddle_n_recall}.

\section{Proof of Theorem~\ref{thm:munonzero}}

For \(s\in\mathcal S\), recall
\[
L_s:=\sum_{i=1}^n s_i\frac{\mu_i}{\sigma_i}.
\]
Define
\begin{equation}
L_*:=\max_{s\in\mathcal S} L_s,
\qquad
\mathcal S_*:=\{s\in\mathcal S:\ L_s=L_*\},
\qquad
m_*:=|\mathcal S_*|.
\label{eq:Lstar_def}
\end{equation}

\paragraph{Ratio test.}
Pick \(s^*\in\mathcal S_*\). From \eqref{eq:exp_saddle_n_recall},
\[
\frac{\exp\!\big(-\Psi(\bu^{(s)}(x))\big)}{\exp\!\big(-\Psi(\bu^{(s^*)}(x))\big)}
=
\exp\!\Big(r(x)(L_s-L_*)+O(1)\Big).
\]
If $s\notin\mathcal{S}_*$, then $L_s-L_*<0$, hence the ratio decays exponentially as $x\to\infty$. Therefore
\begin{equation}
\sum_{s\in\mathcal S}\exp\!\big(-\Psi(\bu^{(s)}(x))\big)
=
\sum_{s\in\mathcal S_*}\exp\!\big(-\Psi(\bu^{(s)}(x))\big)\,\bigl(1+o(1)\bigr).
\label{eq:sum_reduce_to_maximizers}
\end{equation}

\paragraph{Sum over maximizers.}
\begin{equation}
\begin{aligned}
\exp\!\big(-\Psi(\bu^{(s)}(x))\big)
={}&
\exp\!\Bigg\{
-\frac{n}{2}\,r(x)^2+r(x)\,L_s
+\frac{1}{4}\left(
\sum_{i=1}^n\left(\frac{\mu_i}{\sigma_i}\right)^2
-\frac{1}{n}L_s^2
\right) \\
&\quad
+\frac{1}{16\,r(x)}
\left(
\sum_{i=1}^n s_i\left(\frac{\mu_i}{\sigma_i}\right)^3
+\frac{L_s}{n}\sum_{i=1}^n\left(\frac{\mu_i}{\sigma_i}\right)^2
-\frac{2}{n^2}L_s^3
\right)
\Bigg\}
\Big(1+O(r(x)^{-2})\Big).
\end{aligned}
\end{equation}
Hence
\begin{equation}
\begin{aligned}
\sum_{s\in\mathcal S}\exp\!\big(-\Psi(\bu^{(s)}(x))\big)
={}&
m_*\,
\exp\!\Bigg\{
-\frac{n}{2}\,r(x)^2+r(x)\,L_s
+\frac{1}{4}\left(
\sum_{i=1}^n\left(\frac{\mu_i}{\sigma_i}\right)^2
-\frac{1}{n}L_s^2
\right) \\
&\quad
+\frac{1}{16\,r(x)}
\left(
\sum_{i=1}^n s_i\left(\frac{\mu_i}{\sigma_i}\right)^3
+\frac{L_s}{n}\sum_{i=1}^n\left(\frac{\mu_i}{\sigma_i}\right)^2
-\frac{2}{n^2}L_s^3
\right)
\Bigg\}
\Big(1+O(r(x)^{-2})\Big).
\label{eq:sum_approx_by_mstar}
\end{aligned}
\end{equation}

\paragraph{Finally.}
From Step~3,
\[
\int_x^\infty\!\left(\int_{D_s}\exp\!\bigl(-\Phi_w(\tilde{\bu})\bigr)\frac{d\tilde{\bu}}{|u_1\cdots u_{n-1}|}\right)\!dw
=
\frac{A_s(x)}{S_s'(x)}\,e^{-S_s(x)}\bigl(1+O(x^{-2/n})\bigr).
\]
Also,
\[
\frac{A_s(x)}{S_s'(x)}
=
\frac{\pi^{(n-1)/2}\prod_{j=1}^n\sigma_j}{\sqrt n}
\left(\frac{\prod_{j=1}^n\sigma_j}{x}\right)^{1/n}
\left(
1+\frac{n+3}{4n}\,L_s
\left(\frac{\prod_{j=1}^n \sigma_j}{x}\right)^{1/n}
+O(x^{-2/n})
\right).
\]
Therefore, for any $s\in\mathcal S_*$, we have

\begin{equation}
\begin{aligned}
\bF_n(x)
=
\frac{C}{2^{n/2}\sqrt{\pi n}}
\left(\frac{\prod_{j=1}^n \sigma_j}{x}\right)^{1/n}
m_*\,
\exp\!\Bigg\{
-\frac{n}{2}\,\left(\frac{x}{\prod_{j=1}^n \sigma_j}\right)^{2/n}+L_*\left(\frac{x}{\prod_{j=1}^n \sigma_j}\right)^{1/n}
+\frac{1}{4}\left(
\sum_{i=1}^n\left(\frac{\mu_i}{\sigma_i}\right)^2
-\frac{1}{n}L_*^2
\right) \\
+\frac{1}{16}\left(\frac{\prod_{j=1}^n \sigma_j}{x}\right)^{1/n}
\left(
\sum_{i=1}^n s_i\left(\frac{\mu_i}{\sigma_i}\right)^3
+\frac{L_*}{n}\sum_{i=1}^n\left(\frac{\mu_i}{\sigma_i}\right)^2
-\frac{2}{n^2}L_*^3
\right)
\Bigg\}
\left(
1+\frac{n+3}{4n}\,L_*
\left(\frac{\prod_{j=1}^n \sigma_j}{x}\right)^{1/n}
+O(x^{-2/n})
\right).
\end{aligned}
\end{equation}

\paragraph{How to compute \(L_*\) and \(m_*\).}
See Remark~\ref{rem:Lstar}.

\noindent This proves Theorem~\ref{thm:munonzero}.

\section{Illustrations}

We compare the asymptotic approximations from Theorem~\ref{thm:munonzero}
with Monte Carlo estimates of $\mathbb{P}(X_1\cdots X_n>x)$ for the following
three parameter choices:
\[
n=4,\qquad
(\mu_1,\mu_2,\mu_3,\mu_4)=(1.0,\,0.7,\,-0.4,\,1.3),\qquad
(\sigma_1,\sigma_2,\sigma_3,\sigma_4)=(1.0,\,1.2,\,1.5,\,0.9),
\]
\[
n=4,\qquad
(\mu_1,\mu_2,\mu_3,\mu_4)=(1.0,\,0,\,-1.0,\,0),\qquad
(\sigma_1,\sigma_2,\sigma_3,\sigma_4)=(0.8,\,1.2,\,0.3,\,0.9),
\]
and
\[
n=5,\qquad
(\mu_1,\mu_2,\mu_3,\mu_4,\mu_5)=(0.6,\,1.3,\,1.5,\,0.2,\,0.4),\qquad
(\sigma_1,\sigma_2,\sigma_3,\sigma_4,\sigma_5)=(0.8,\,0.7,\,0.9,\,0.6,\,0.8).
\]

\begin{figure}[H]
    \centering
    \includegraphics[width=0.75\textwidth]{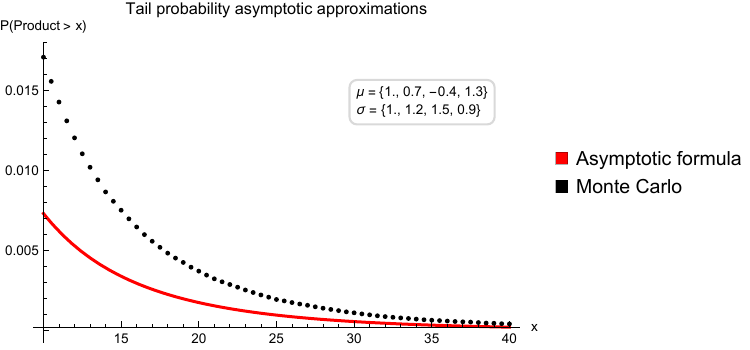}
    \caption{Monte Carlo estimates and asymptotic approximations for
    $n=4$, $(\mu_1,\mu_2,\mu_3,\mu_4)=(1.0,0.7,-0.4,1.3)$, and
    $(\sigma_1,\sigma_2,\sigma_3,\sigma_4)=(1.0,1.2,1.5,0.9)$.}
    \label{fig:illustration1}
\end{figure}

\begin{figure}[H]
    \centering
    \includegraphics[width=0.75\textwidth]{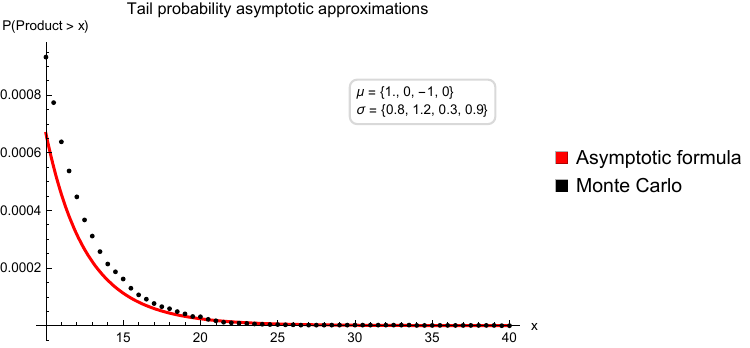}
    \caption{Monte Carlo estimates and asymptotic approximations for
    $n=4$, $(\mu_1,\mu_2,\mu_3,\mu_4)=(1.0,0,-1.0,0)$, and
    $(\sigma_1,\sigma_2,\sigma_3,\sigma_4)=(0.8,1.2,0.3,0.9)$.}
    \label{fig:illustration2}
\end{figure}

\begin{figure}[H]
    \centering
    \includegraphics[width=0.75\textwidth]{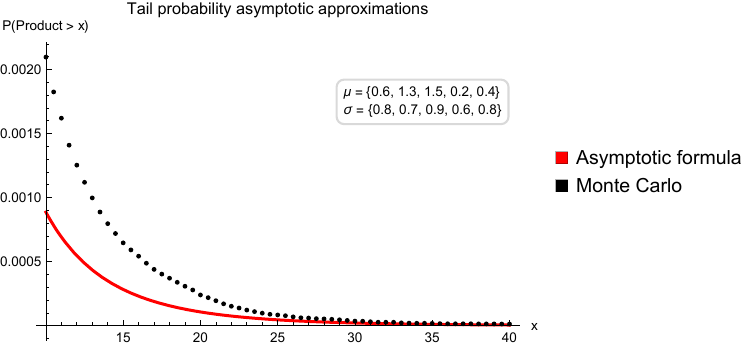}
    \caption{Monte Carlo estimates and asymptotic approximations for
    $n=5$, $(\mu_1,\mu_2,\mu_3,\mu_4,\mu_5)=(0.6,1.3,1.5,0.2,0.4)$, and
    $(\sigma_1,\sigma_2,\sigma_3,\sigma_4,\sigma_5)=(0.8,0.7,0.9,0.6,0.8)$.}
    \label{fig:illustration3}
\end{figure}

\bibliographystyle{plain}
\bibliography{bibliography}

\end{document}